\documentclass[english]{amsart}
\usepackage{amssymb,amsmath}
\usepackage{caption}
\usepackage{float}
\usepackage{graphicx,subfigure,epsfig}
\usepackage{epstopdf}
\usepackage{mathtools}
\usepackage{multicol}
\usepackage{fancybox}
\usepackage[usenames, dvipsnames]{color}
\newtheorem{theorem}{Theorem}[section]
\newtheorem{definition}{Definition}[section]
\newtheorem{lemma}{Lemma}[section]
\newtheorem{remark}{Remark}[section]

\newtheorem{corollary}{Corollary}[section]
\newtheorem{example}{Example}[section]
\newtheorem{proposition}{Proposition}[section]
\newtheorem{conjecture}{Conjecture}[section]

\newcommand{\restr}[1]{|_{#1}}

\begin{document}
\title{Extending  quasi-alternating links}

\author{Nafaa Chbili}
\address{Department of Mathematical Sciences\\College of Science UAE University\\15551 Al Ain, U.A.E.}
\email{nafaachbili@uaeu.ac.ae}
\author{Kirandeep Kaur}
\address{Department of Mathematical Sciences\\College of Science UAE University\\15551 Al Ain, U.A.E.}
\email{kirandeep@uaeu.ac.ae}
\thanks{ This research was funded  by  United Arab Emirates University, UPAR grant $\#$G00002650.}

\begin{abstract}
Champanerkar and Kofman \cite{champanerkar2009twisting} introduced an interesting way to construct new examples of quasi-alternating links from existing ones. Actually, they
 proved that replacing a quasi-alternating crossing $c$ in a quasi-alternating link by a rational tangle of same type  yields a new quasi-alternating link.
This construction has been extended to alternating algebraic tangles and  applied  to  characterize all quasi-alternating Montesinos links.
In this paper, we extend this technique  to any alternating tangle of same type as $c$.  As an application, we give new examples of quasi-alternating knots of 13 and 14 crossings.
Moreover, we prove that the Jones polynomial of a quasi-alternating link that is  obtained in this way  has no gap if the original link has no gap in its Jones polynomial. This  supports a conjecture introduced in   \cite{chbili2019jones}, which states that the Jones polynomial of any prime quasi-alternating link except $(2,n)$-torus link has no gap.\\

\end{abstract}
\makeatletter{\renewcommand*{\@makefnmark}{}
\makeatother}

\maketitle
\section{Introduction}
Let $L$ be a non-split  alternating link in the three-sphere $S^3$ and $\Sigma(L)$ be  the branched double-cover of  $S^3$  along $L$. Ozsv\'{a}th and Szab\'{o} \cite{ozsvath2005heegaard} proved that the Heegaard Floer homology of $\Sigma(L)$   is determined  entirely  by   the determinant of the link, $\det(L)$.
Furthermore, they noticed that this interesting homological property extends to a wider family of links that they called quasi-alternating. These links are     defined in the following  recursive way.

\begin{definition} The set $\mathcal{Q}$ of quasi-alternating links is the smallest set satisfying the following properties:
\begin{itemize}
\item  The unknot is in $\mathcal{Q}$.
\item If a link $L$ has a diagram  containing a crossing $c$ such that
\begin{enumerate}
\item both smoothings of the diagram of $L$ at the crossing $c$, $L_0$ and $L_\infty$, as in Fig.~\ref{1},
belong to $\mathcal{Q}$,
\item $\det(L_0)$, $\det(L_\infty) \geq 1$ and  $\det(L) = \det(L_0) + \det(L_\infty)$;
 \end{enumerate}
then $L$ is in $\mathcal{Q}$ and we say that   $L$ is quasi-alternating at the crossing $c$.
\end{itemize}
\end{definition}

 With an elementary induction on the determinant of the link, this definition can be used to prove that non-split alternating links belong to $\mathcal{Q}$  \cite{ozsvath2005heegaard}. The knots $8_{20}$ and $8_{21}$ are the first examples, in the knot table,  of  non-alternating quasi-alternating knots.
  In general,  it is not   easy  to decide  whether a given link is quasi-alternating only by using this recursive  definition.
Throughout  the past years, several obstruction criteria for a link to be  quasi-alternating  have  been introduced. For instance, quasi-alternating links are proved to be  homologically-thin in  both Khovanov homology and link  Floer homology   \cite{manolescu2007khovanov}.  In addition, the reduced odd Khovanov homology group of any quasi-alternating link is thin and torsion-free \cite{ozsvath2013odd}. On the other hand, the branched double-cover of any quasi-alternating link is an $L$-space \cite{ozsvath2005heegaard}, and the signature of an oriented quasi-alternating  link $L$ equals minus four times the Heegaard Floer correction term of $(\sum(L), t_0)$ \cite{lisca2015signatures}.
 The behavior of polynomial invariants of quasi-alternating links has been also  investigated  leading to obstruction criteria  involving the Q-polynomial and the two-variable
Kauffman polynomial  \cite{qazaqzeh2015new, teragaito2014quasi, teragaito2015quasi}.\\
In \cite{champanerkar2009twisting}, Champanerkar and  Kofman  introduced an interesting way to construct  new quasi-alternating links from existing ones. They proved that a  new quasi-alternating link can be obtained by
 replacing a quasi-alternating crossing by an alternating rational tangle of same type. This construction has been generalized to algebraic tangles in \cite{qazaqzehchbiliQublan20014} and applied to the classification of quasi-alternating Montesionos links.\\
In this paper, we prove  that if a quasi-alternating crossing is replaced by any alternating connected tangle of same type, whether  algebraic or non-algebraic, then the resulting link is quasi-alternating at every new non-nugatory crossing.
This construction is illustrated by the example in Fig.~\ref{ne1} where the quasi-alternating crossing $c$ in the diagram of pretzel knot  $P(2,1,-3)$  is replaced by a non algebraic tangle to obtain a new quasi-alternating link $L10n16$.
 This extends the previous  work of  Champanerkar-Kofman\cite{champanerkar2009twisting} and its generalization in \cite{qazaqzehchbiliQublan20014}. As an application of this construction,
  we list few quasi-alternating knots of 13 and  14 crossings, see Section 3.\\

\begin{figure}[!ht] {\centering
 \subfigure[$P(2,1,-3)$] {\includegraphics[scale=1.2]{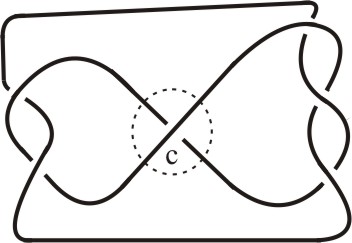} } \hspace{2cm}
\subfigure[$L10n16$] {\includegraphics[scale=1.2]{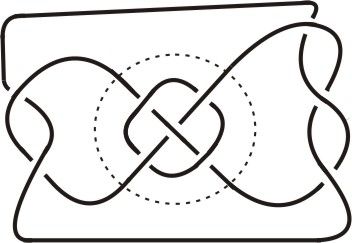}}
\caption{A link obtained from pretzel knot $P(2,1,-3)=5_2$ by extending quasi-alternating crossing $c$ to a non-algebriac tangle}\label{ne1}}
\end{figure}
On the other hand, it is well-known that the Jones polynomial of  an alternating link satisfies interesting properties that have been applied to solve some challenging conjectures in classical knot theory  \cite{kauffman1987state,Murasugi1987alternating,thistlethwaite1987spanning}.
More precisely,   the Jones polynomial of  any prime alternating link which is not a $(2,n)$-torus link is alternating and has no gap. In other words,  if $V_L(t)= \sum_{i=n}^{m}c_it^i$, then its coefficients satisfy $c_ic_{i+1}<0$ for $n\leq i \leq m-1$.
In the case of a quasi-alternating link, the Jones polynomial is alternating as a consequence of that the  reduced  Khovanov homology group of any quasi-alternating link is thin \cite{manolescu2007khovanov}.
 Further, in \cite{chbili2019jones}  Chbili and Qazaqzeh studied the behavior of the  Jones polynomial in case of quasi-alternating links and suggested  the following conjecture:
\begin{conjecture}\label{conj}
If $L$ is a prime quasi-alternating link which is not a $(2,n)$-torus link, then the coefficients of the Jones polynomial of $L$ satisfy $c_ic_{i+1}<0$ for all $n \leq i\leq m-1$.
\end{conjecture}
This conjecture has been confirmed for all quasi-alternating links of braid index 3 and for all quasi-alternating Montesinos links. It was also proved that if a quasi-alternating link satisfies this conjecture then any link obtained by Champernkar-Kofman twisting
construction will also satisfy   the conjecture.  In this paper, we prove that the coefficients of the Jones polynomial of any quasi-alternating link that is obtained by replacing  a quasi-alternating crossing by an  alternating tangle, satisfies the condition given in Conjecture~\ref{conj} whenever the Jones polynomial of the original quasi-alternating  link has no gap.

The present  paper is organized as follows. In Section 2, we introduce some background necessary for the rest of the paper. In Section 3, we prove  our main result, Theorem~\ref{Thm1}, and illustrate that by listing some examples of quasi-alternating knots.   Section 4 will be devoted  to confirm that Conjecture~\ref{conj} holds for links obtained by our construction.

\section{Background and notations }
At the beginning of this section, we shall   recall some basic definitions and results about  quasi-alternating links and Jones polynomials which will be  needed in  the rest of  this paper.
The Jones polynomial  \cite{jones1985polynomial} is an invariant of ambient isotopy of oriented links in $S^3$. It is a Laurent polynomial with integral coefficients which  can be defined in several equivalent ways.
 In this paper, we shall use the  Kauffman state model definition of the Jones polynomial \cite{kauffman1987state}. The Kauffman bracket polynomial $<L> \in \mathbb{Z}[A,A^{-1}]$ is an invariant of regular isotopy of links which
 can be defined on link  diagrams by the following relations:
\begin{enumerate}
\item $<O>=1$ and $<O\cup L>=(-A^2-A^{-2})<L>$;
\item If $L$, $L_0$ and $L_\infty$ are 3 link diagrams which differ only in a local region where  they look as  illustrated in Fig.~\ref{1}, then
\[<L>=A<L_0>+A^{-1}<L_\infty>.\]
\end{enumerate}
 \begin{figure}[!ht] {\centering
 \subfigure[$L$]{\includegraphics[scale=0.8]{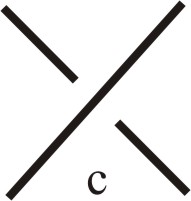} } \hspace{1cm}
\subfigure[$L_0$] {\includegraphics[scale=0.8]{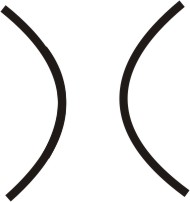}}\hspace{1cm}
\subfigure[$L_\infty$] {\includegraphics[scale=0.8]{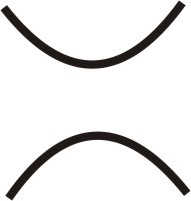}}
\caption{A link $L$ and its smoothings at crossing $c$, $L_0$ and $L_\infty$.}\label{1}}\end{figure}

Kauffman proved that the  bracket polynomial of a link diagram $D$ can be expressed as   a state summation. A state $S$ of a link diagram $D$ is a way to perform smoothings
 at each crossing of $D$. If $\alpha(S)$ and $\beta(S)$ are the number of $A$ and $A^{-1}$-smoothings
 in  the  state $S$, then the bracket polynomial is given by
\[<D>=\displaystyle \sum_{S}A^{\alpha(S)}A^{\beta(S)}(-A^2-A^{-2})^{\sharp(S)-1},\]
where the summation is taken through all the states of $D$ and $\sharp(S)$ is the number of circles obtained after smoothing all the crossings.

 The Jones Polynomial $V_{L}(t)\in\mathbb{Z}[t^{\frac{1}{2}},t^{\frac{-1}{2}}]$, of an oriented link $L$ is defined by
 $$V_{L}(t)=(-A^3)^{-w(L)}<L> \restr{A^{4}=t^{-1}},$$
  where $w(L)$ is the writhe of $L$ defined as the algebraic  sum of  signs of the crossings of $L$,   and
 $<L>$ denotes the bracket polynomial of the link obtained by forgetting the orientation of $L$.

Obviously,  there is no gap in the Jones polynomial if and only if there is a gap of exactly four between any two consecutive powers of $A$ in the Kauffman bracket polynomial.
  For simplicity, we call a polynomial $f(A)$  alternating if powers of $A$ in $f(A)$  differ by multiples of four and coefficients of consecutive terms are of opposite signs.
  Moreover, we say that $f(A)$ is strictly alternating if the gap between any two consecutive terms is exactly four.\\

 A crossing $c$ in a link diagram $L$ is said to be \emph{nugatory} if it looks as pictured in  Fig.~\ref{nt}.

 \begin{figure}[!ht] {\centering
 \includegraphics[scale=.45]{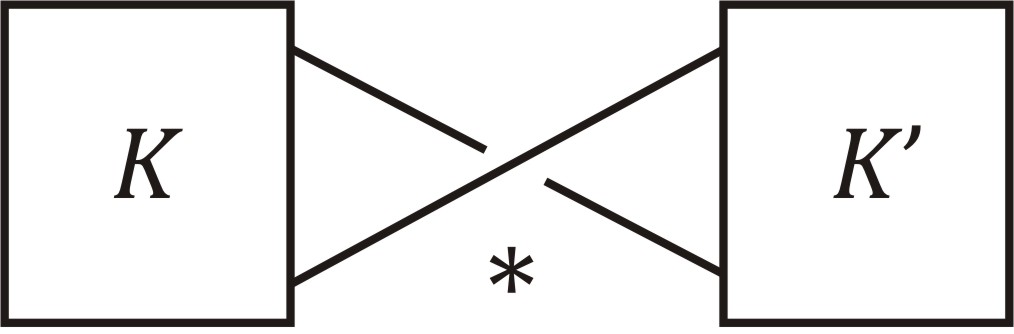}
\caption{Nugatory crossing}\label{nt}}\end{figure}
\begin{remark}\label{remk1} Let $c$ be a nugatory crossing of a link diagram $L$. Then $<L_0>=(-A^2-A^{-2})<L_\infty>$ if $L_0$ is a split link and $<L_\infty>=(-A^2-A^{-2})<L_0>$ if $L_\infty$ is a split link.
\end{remark}

A tangle is a proper  embedding of disjoint union of 2 arcs and possibly  loops into a three-ball $B^3$ such that the end points of these arcs lay on 4 fixed points of the boundary of $B^3$.
Tangles are represented by their planar diagrams and are considered up to isotopy of $B^3$ keeping the boundary fixed. We say that a tangle is connected if it is represented by a connected diagram.
Throughout this paper, unless otherwise stated, we are only considering connected tangles.
Given a tangle $T$, one can join the end points of $T$ by simple arcs to define the numerator closure $N(T)$ and the denominator  closure $D(T)$.
 We call a crossing $c$ of a tangle $T$  nugatory in $T$ if it is  nugatory in both the closures of $T$. While a non-nugatory crossing in a tangle $T$ is said to be trivial if it is a nugatory crossing either in the numerator closure $N(T)$ or in the denominator closure $D(T)$ of $T$. For example, the crossing denoted by $*$ in Fig~\ref{nt2}(a) is a nugatory crossing. However, the crossing in Fig~\ref{nt2}(b) is a trivial crossing.
The crossings in a tangle which are neither nugatory nor trivial are called  non-trivial crossing.
\begin{figure}[!ht] {\centering
 \subfigure[]{\includegraphics[scale=.55]{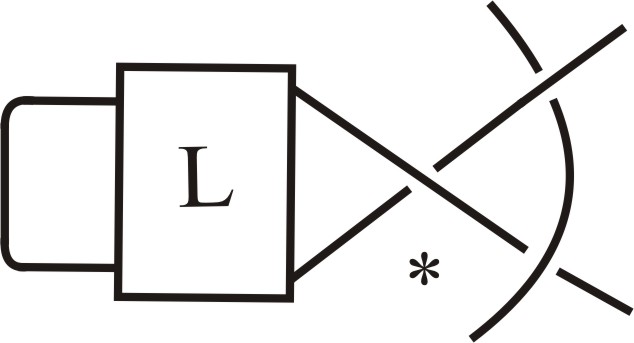}} \hspace{1cm}
 \subfigure[]{\includegraphics[scale=.55]{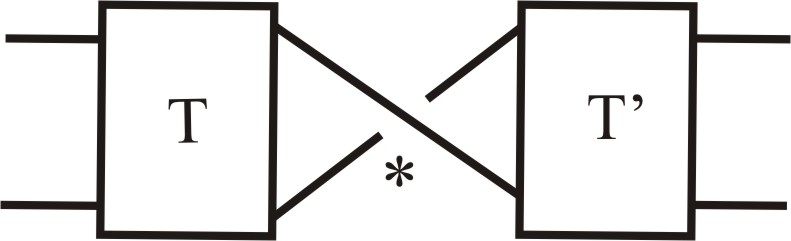}}
\caption{Nugatory crossing and trivial crossing of a tangle diagram}\label{nt2}}\end{figure}
\begin{remark}\label{remk2}
Non-trivial crossings of an alternating  tangle $T$ are quasi-alternating in both the closures of $T$. A trivial crossing of an alternating  tangle $T$ is  quasi-alternating in one of the closures of $T$.
\end{remark}
The determinant of a link $L$, $\det(L)$, is a well-known numerical invariant of links  that  can be defined from the Jones polynomial by $\det(L)=|V_L(-1)|$.
This invariant  can  be computed by considering  the number of spanning trees in a checkerboard graph of a
 link projection. Let us now  recall this computation  method.
\begin{figure}[!ht] {\centering\includegraphics[scale=.8]{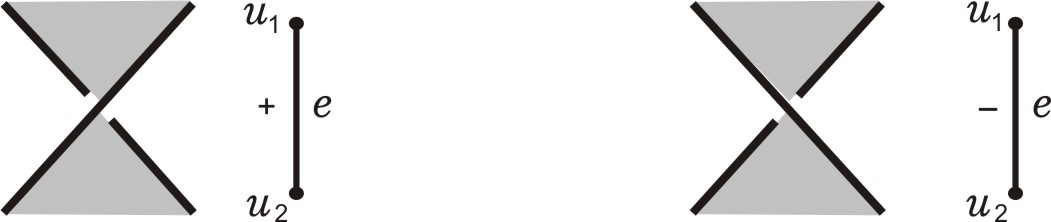}
\caption{}\label{2}}\end{figure}

A Tait graph, $\mathcal{G}(L)$, of a connected link diagram $L$ is a  graph which is obtained by coloring the diagram in  checkerboard fashion, followed by assigning a vertex to every shaded region and an edge to every crossing with a sign as illustrated in Fig.~\ref{2}.
 Then, the determinant of the  link $L$ is given as follows:

\begin{lemma}[\cite{champanerkar2009twisting}]For any spanning tree $\mathsf{T}$ of $\mathcal{G}(L)$, let $v(\mathsf{T})$ be the number of positive edges in $\mathsf{T}$. Let $s_{v}(L)=\sharp\{\text{spanning trees $\mathsf{T}$ of~} \mathcal{G}(L)\mid v(\mathsf{T})=v \}$. Then
\[\det(L)=\Big|\displaystyle \sum_v (-1)^vs_v(L)\Big|.\]
\end{lemma}

\begin{remark} Consider a quasi-alternating link $L$ and a quasi-alternating crossing $c$ of $L$. If $e$ is the corresponding edge in $\mathcal{G}(L)$, then there is one-to-one correspondence between spanning trees of $\mathcal{G}(L)$ that contain $e$ (respectively, does not contain $e$) and spanning trees of $\mathcal{G}(L_0)$ (respectively, $\mathcal{G}(L_\infty)$). Therefore,
\[  s_{v}(L_\infty)=\sharp\{\text{spanning trees $\mathsf{T}$ of~} \mathcal{G}(L)\mid e\notin \mathsf{T} \text{~and~} v(\mathsf{T})=v \} \]
\[s_{v-1}(L_0)=\sharp\{\text{spanning trees $\mathsf{T}$ of~} \mathcal{G}(L)\mid e\in \mathsf{T} \text{~and~} v(\mathsf{T})=v \},\]  \text{and }

\begin{equation} \label{eq1} \det(L)=\Big|\displaystyle \sum_v (-1)^vs_v(L)\Big|=\Big|\displaystyle \sum_v (-1)^vs_{(v-1)}(L_0)+\displaystyle \sum_v (-1)^vs_v(L_\infty)\Big|.
\end{equation}

 By definition of a quasi-alternating crossing, $\det(L)=\det(L_0)+\det(L_{\infty})$ and hence from equation  (\ref{eq1}) above we have
\[\Big(\displaystyle \sum_v (-1)^vs_{(v-1)}(L_0)\Big).\Big(\displaystyle \sum_v (-1)^vs_v(L_\infty)\Big)>0.\]
\end{remark}

\section{The Main Theorem}
 Before stating  the main result in  this section, let us introduce  some basic definitions and notations.
 A spanning tree of a connected planar graph $\mathcal{G}$, is a connected, acyclic sub-graph that contains all  vertices of $\mathcal{G}$. Given a graph $\mathcal{G}$, we  denote by $\tau(\mathcal{G})$  the number of spanning trees in  $\mathcal{G}$.
\begin{definition} A sub-graph $\mathcal{H}$ of a connected planar graph $\mathcal{G}$  is said to be almost spanning tree with respect to vertices $u_1$ and $u_2$, if there is no path between $u_1$ and $u_2$, and by adding one edge it becomes a spanning tree of $\mathcal{G}$.
\end{definition}
\begin{definition} An alternating tangle $T$ is said to be positive (respectively,  negative) if the edges of Tait graph $\mathcal{G}(T)$, with induced checkerboard coloring as shown in Fig.~\ref{2}, are positive (respectively,  negative).
\end{definition}
Let $e$ be the edge corresponding to a crossing $c$ with checkerboard coloring as illustrated in Fig.~\ref{2}. We say that an alternating tangle $T$ extends the  crossing $c$ (or a tangle of same type as $c$),
 if $T$ is positive whenever  $e$ is positive and negative whenever  $e$ is negative. The tangle shown in Fig.~\ref{3} is a positive tangle.
 \begin{figure}[!ht] {\centering\includegraphics[scale=.8]{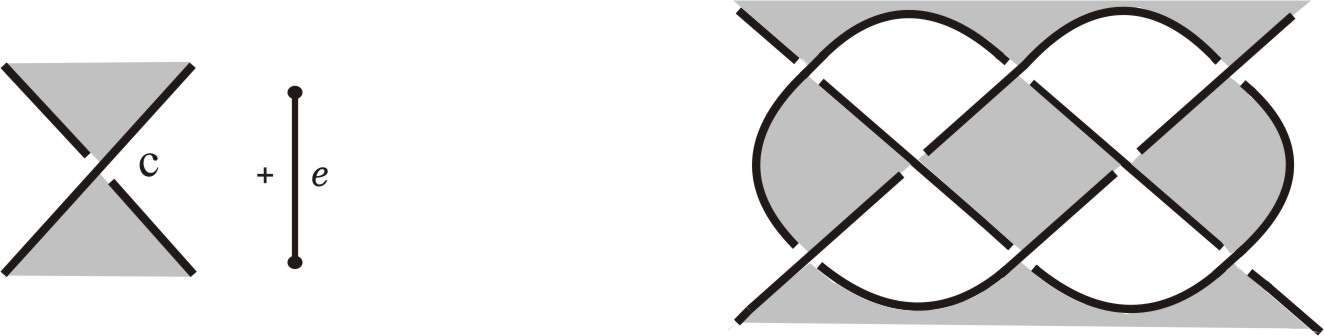}
\caption{A tangle extending  the crossing $c$}\label{3}}\end{figure}

\begin{lemma}\label{ext:lem}
Let $L'$ be a link obtained from a quasi-alternating link diagram $L$ by replacing  a quasi-alternating crossing $c$ by a reduced alternating tangle $T$ extending  $c$. Then at every crossing of $T$ in $L'$ we have,
\[\det(L')=\det(L'_{0})+\det(L'_{\infty}).\]
\end{lemma}
\begin{proof}
Let $c$ be a quasi-alternating crossing of a link diagram $L$ and $e$ be its corresponding edge in $\mathcal{G}(L)$.
Let $T$ be a tangle which extends   crossing $c$ and $L'$ be the  link obtained from $L$ by replacing   $c$ by  tangle $T$.

Choose the checkerboard coloring of $L$ such that the edge $e(=u_1u_2)$ in $\mathcal{G}(L)$ is positive. Let $\mathcal{G}(L')$ be the Tait graph of $L'$ with the induced checkerboard coloring, and $\mathcal{G}(T)$ be the sub-graph $\mathcal{G}(L')\restr{T}$ of $\mathcal{G}(L')$. Then all the edges of $\mathcal{G}(T)$ are positive.
Assume that $\mathcal{G}(T)$ has $n+2$  vertices, and let $x=\tau(\mathcal{G}(T))$ and $y$ be the number of almost spanning trees of $\mathcal{G}(T)$ with respect to $u_1$ and $u_2$.

Let $c'$ be a crossing of $L'$ that belongs to $T$ and $e'$ be its corresponding edge in $\mathcal{G}(T)$.
Suppose that $x_{e'}$ and $y_{e'}$ are the number of spanning trees and almost spanning trees of $\mathcal{G}(T)$ which contains edge $e'$, respectively.

Then, for each spanning tree $\mathsf{T}$ of $\mathcal{G}(L)$ such that $e \in \mathsf{T}$, there exists
 $xs_{v(\mathsf{T})-1}(L_0)$ spanning trees in $\mathcal{G}(L')$
such that $u_1$ and $u_2$ are connected through sub-graph $\mathcal{G}(T)$ of $\mathcal{G}(L')$, more precisely, through spanning trees of $\mathcal{G}(T)$. Out of these, there are  $x_{e'}s_{v(\mathsf{T})-1}(L_0)$  spanning trees of $\mathcal{G}(L')$ which contain edge $e'$.

Since $\mathcal{G}(T)$ has $n+2$  vertices, every spanning tree of $\mathcal{G}(T)$  has $n+1$ edges. Therefore, every spanning tree of $\mathcal{G}(L')$ corresponding to a spanning tree $e\in \mathsf{T}$ of $\mathcal{G}(L)$ has $v(\mathsf{T})+n$ positive edges.

Similarly, for each spanning tree $\mathsf{T}$ of $\mathcal{G}(L)$ such that $e \notin \mathsf{T}$, there exists  $ys_{v(\mathsf{T})}(L_\infty)$ spanning trees in $\mathcal{G}(L')$
such that $u_1$ and $u_2$ are not connected through sub-graph $\mathcal{G}(T)$ of $\mathcal{G}(L')$. More precisely, every spanning tree of $\mathcal{G}(L')$ that corresponds to a spanning tree $e\notin\mathsf{T}$ of $\mathcal{G}(L)$ is a union of $\mathsf{T}$ with an almost spanning tree of $\mathcal{G}(T)$ with respect to $u_1$ and $u_2$. Out of these, there are   $y_{e'}s_{v(\mathsf{T})}(L_\infty)$  spanning trees of $\mathcal{G}(L')$ which contain edge $e'$.

Since a spanning tree of $\mathcal{G}(T)$ has $n+1$ edges, an almost spanning tree of $\mathcal{G}(T)$ has $n$ edges. Therefore, every spanning tree of $\mathcal{G}(L')$ that corresponds to a spanning tree $e\notin \mathsf{T}$ of $\mathcal{G}(L)$ has $v(\mathsf{T})+n$  positive edges.

Hence, for every $v'$ such that $v'=v+n$, we have
\begin{equation}\label{Thm:eq1}
\begin{split}
 s_{v'}(L')&=xs_{v-1}(L_0)+ys_{v}(L_\infty)\\
\Big| \displaystyle \sum_{v'}(-1)^{v'}s_{v'}(L') \Big|&=\Big| \displaystyle \sum_{v'}(-1)^{v'} (xs_{v-1}(L_0)+ys_{v}(L_\infty) )\Big|\\
&=\Big| \displaystyle \sum_{v}(-1)^{v} (xs_{v-1}(L_0)+ys_{v}(L_\infty))\Big|
\end{split}
\end{equation}
\begin{equation}\label{Thm:eq2}
\begin{split}
 s_{v'-1}(L'_{0})&=x_{e'}s_{v-1}(L_0)+y_{e'}s_{v}(L_\infty)\\
\Big| \displaystyle \sum_{v'}(-1)^{v'-1}s_{v'-1}(L'_0)\Big|&=\Big|\displaystyle \sum_{v'}(-1)^{v'-1} (x_{e'}s_{v-1}(L_0)+y_{e'}s_{v}(L_\infty))\Big| \\
&=\Big|\displaystyle \sum_{v}(-1)^{v} (x_{e'}s_{v-1}(L_0)+y_{e'}s_{v}(L_\infty))\Big|,
\end{split}
\end{equation}
\begin{equation}\label{Thm:eq3}
\begin{split}
s_{v'}(L'_{\infty})&=(x-x_{e'})s_{v-1}(L_0)+(y-y_{e'})s_{v}(L_\infty)\\
\Big| \displaystyle \sum_{v'}(-1)^{v'}s_{v'}(L'_{\infty})\Big|&= \Big|\displaystyle \sum_{v'}(-1)^{v'} ((x-x_{e'})s_{v-1}(L_0)+(y-y_{e'})s_{v}(L_\infty))\Big|\\
&=\Big|\displaystyle \sum_{v}(-1)^{v} ((x-x_{e'})s_{v-1}(L_0)+(y-y_{e'})s_{v}(L_\infty))\Big|,
\end{split}
\end{equation}

where $L'_0$ and $L'_\infty$ are links obtained from $L'$ by smoothing crossing $c'$.

 Since $L$ is quasi-alternating at $c$, $\displaystyle \sum_{v}(-1)^{v} s_{(v-1)}(L_0). \displaystyle \sum_{v}(-1)^{v}s_{v}(L_\infty)>0$. Thus from equations (\ref{Thm:eq1}),~(\ref{Thm:eq2}) and ~(\ref{Thm:eq3}) above, we get
\begin{align*}
\begin{split}
\det(L')&=\Big|\displaystyle \sum_{v'}(-1)^{v'}s_{v'}(L')\Big|= x \det(L_0)+y \det(L_\infty),\\
\det(L'_{0})&=\Big|\displaystyle \sum_{v'}(-1)^{v'-1}s_{v'-1}(L'_{0})\Big|= x_{e'} \det(L_0)+y_{e'} \det(L_\infty),\\
\det(L'_{\infty})&=\Big|\displaystyle \sum_{v'}(-1)^{v'}s_{v'}(L'_{\infty})\Big|= (x-x_{e'}) \det(L_0)+(y-y_{e'}) \det(L_\infty),
\end{split}
\end{align*}
and hence $\det(L')=\det(L'_{0})+\det(L'_{\infty})$.
\end{proof}

Now, we state the following lemma which will be needed in the proof of Theorem~\ref{Thm1}. The result  follows  directly  from Lemma 2.3 of \cite{champanerkar2009twisting}.
\begin{lemma}\label{lem:con} If $L_1$ and $L_2$ are two quasi-alternating links, then any connected sum of $L_1$ and  $L_2$ is  quasi-alternating.
\end{lemma}

Our main result in this paper can be stated as follows.
\begin{theorem}\label{Thm1}
Let $L$ be a quasi-alternating link diagram at a crossing $c$.
Let $L'$ be the  link obtained from  $L$ by replacing $c$ by a reduced alternating tangle $T$ that extends $c$. Then $L'$ is  quasi-alternating at every crossing of $T$.
\end{theorem}
Before proving Theorem~\ref{Thm1}, we illustrate our result by the following example. More examples are to be listed at the end of this section.
\begin{example}
The knot $13n695$ can be obtained from $8_{20}$
 by replacing the quasi-alternating crossing $c$ by the  non-algebraic tangle $T$  in Fig.~\ref{e1}.
By Theorem~\ref{Thm1},  $13n695$  is  quasi-alternating.\\
\begin{figure}[!ht] {\centering
 \subfigure {\includegraphics[scale=0.9]{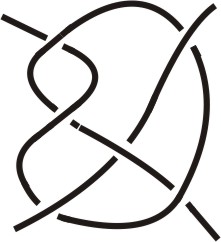} } \hspace{2cm}
\subfigure {\includegraphics[scale=0.9]{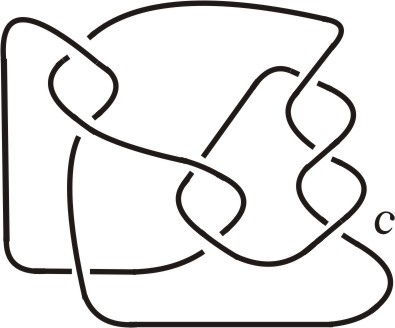}}
\caption{A non-algebraic tangle $T$ and a diagram of $8_{20}$ knot.}\label{e1}}
\end{figure}
\end{example}

\begin{proof}
We prove this result by induction  on the number of crossings in the  tangle $T$. If $T$ has only one crossing, then link $L'$ is  nothing other than  $L$ and the result is obvious.
Assume now that the result is true  for every  tangle  with less than $n$ crossings which extends $c$.\\
Let $T$ be a tangle with $n$  crossings which extends $c$.
If $T$ is a twisted tangle, then the result is true by Theorem 2.1 of \cite{champanerkar2009twisting}. Otherwise, let $c'$ be an arbitrary crossing of $T$. By Lemma~\ref{ext:lem}, the determinant
 condition is satisfied at crossing $c'$, i.e.,
\[\det(L')=\det(L'_{0})+\det(L'_{\infty}).\]
Notice that links $L'_{0}$ and $L'_{\infty}$ are obtained from $L$ by replacing  crossing $c$ by  tangles $T_0$ and $T_{\infty}$, where $T_0$ and $T_{\infty}$ are the tangles obtained from $T$ by smoothing crossing $c'$. Reduce these tangles by applying Reidemeister move $RI$, if needed,  to obtain reduced tangles  $T'_0$ and $T'_{\infty}$.\\
It is easy to observe that if $c'$ is a non-trivial crossing, then $T'_0$ and $T'_\infty$ are extended tangles of crossing $c$. Therefore, $L'_{0}$ and $L'_{\infty}$ are equivalent to link diagrams which are obtained from $L$ by replacing crossing $c$ by  tangles $T'_0$ and $T'_{\infty}$, respectively.  As numbers of crossings in $T'_0$ and $T'_{\infty}$ are both less than $n$, then  links $L'_{0}$ and $L'_{\infty}$ are quasi-alternating.

If $c'$ is a trivial crossing, then either $T'_0$ or $T'_\infty$ is  not a connected tangle and the other one is an extended tangle of $c$ with less than $n$ crossings. If the tangle is an extended tangle of crossing $c$, then the corresponding link is obtained from $L$ by replacing  crossing $c$ by a  tangle having  less than $n$ crossings.
Otherwise, the link is a connected sum of $L_0$ if $T'_0$ is not connected or $L_\infty$ if $T'_\infty$ is not connected,  with two alternating link diagrams.
Using Lemma~\ref{lem:con} and induction hypothesis, $L'_0$ and $L'_{\infty}$ are quasi-alternating links in both  cases.
Hence, $L'$ is a quasi-alternating link. Since $c'$ is an arbitrary crossing of $T$, $L'$ is quasi-alternating at every crossing of tangle $T$. This ends the proof of Theorem~\ref{Thm1}.
 \end{proof}

\vspace{5mm}

 \begin{figure}[!ht] {\centering
 \subfigure[$T_{pqrs}$] {\includegraphics[scale=1]{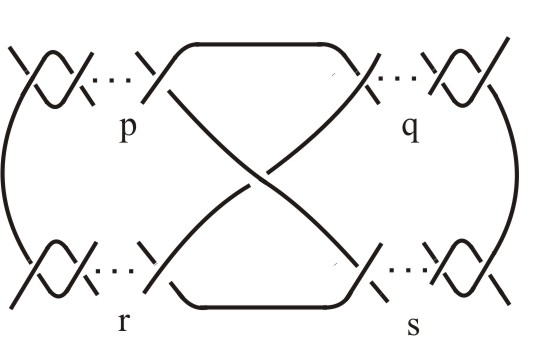} } \hspace{1cm}
\subfigure[$\widehat{T}_{pqrs}$] {\hspace{1cm}\includegraphics[scale=1]{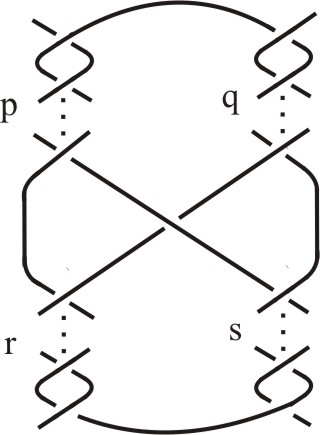}\hspace{1cm}}\caption{ Non-algebraic tangle diagrams.}\label{tangles}}
\end{figure}
 \begin{figure}[!ht] {\centering
 \subfigure[$8_{21}$] {\includegraphics[scale=0.6]{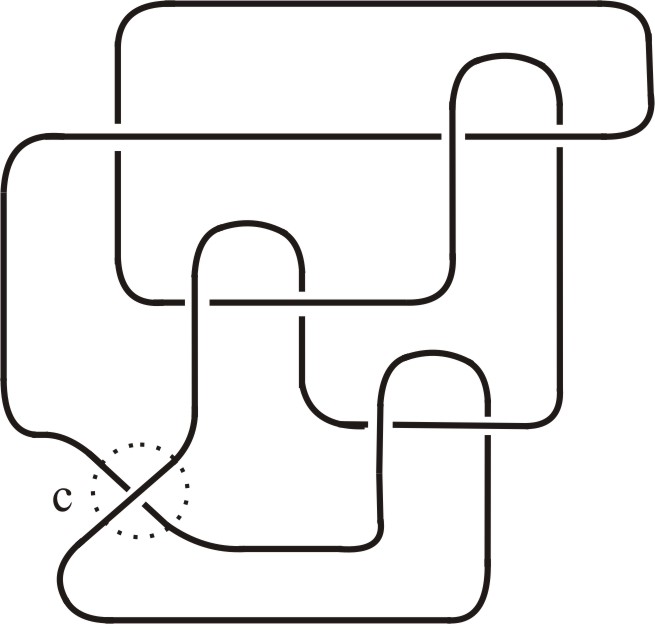} } \hspace{1cm}
\subfigure[$(9_{47})^*$] {\hspace{1cm}\includegraphics[scale=1]{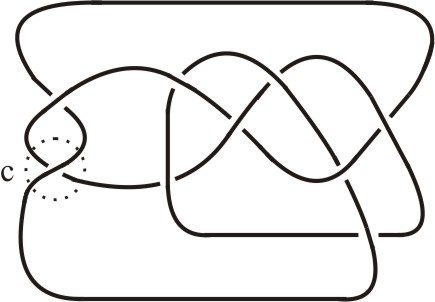}\hspace{1cm}}
\caption{ Quasi-alternating diagrams of the knots $8_{21}$ and the  mirror image of the knot $9_{47}$.}\label{8_21 and 9_47}}
\end{figure}
\begin{example}
With the help of the construction in  Theorem~\ref{Thm1}, we list few quasi-alternating knots of 13 and 14 crossings. Let $K (T)_c$ be the knot which is obtained from a knot diagram $K$ by replacing a crossing $c$ by a tangle $T$.
 For  $p,q,r,s\in \mathbb{Z}$, consider the non-algebraic tangles   $T_{pqrs}$ and $\widehat{T}_{pqrs}$   displayed  in Fig.~\ref{tangles}.
By applying Theorem~\ref{Thm1} to  quasi-alternating diagrams of the knots  $8_{21}$ and $9_{47}$ at the indicated crossings as   shown in Fig.~\ref{8_21 and 9_47}, we obtain  the following table of quasi-alternating knots.

\begin{table}[ht]\begin{center}
\begin{tabular}{ c|c }
Knot listed in knot atlas & Quasi-alternating diagram\\
\hline
13n329 & $8_{21}(T)_c$, where $T=\widehat{T}_{1112}$\\
\hline
13n351 & $8_{21}(T)_c$, where $T=\widehat{T}_{2111}$\\
\hline
13n513 & $8_{21}(T)_c$, where $T=\widehat{T}_{1121}$\\
\hline
13n525 & $8_{21}(T)_c$, where $T=\widehat{T}_{1211}$ \\
\hline
13n344 & $8_{21}(T)_c$, where $T=T_{1112}$\\
\hline
13n357 & $8_{21}(T)_c$, where $T=T_{2111}$\\
\hline
13n520 & $8_{21}(T)_c$, where $T=T_{1121}$\\
\hline
13n530 & $8_{21}(T)_c$, where $T=\widehat{T}_{1211}$ \\
\hline
14n20077 & $(9_{47})^*(T)_c$, where $T=\widehat{T}_{1211}$\\
\hline
14n20087 & $(9_{47})^*(T)_c$, where $T=T_{1211}$\\
\hline
14n20088 & $(9_{47})^*(T)_c$, where $T=\widehat{T}_{1121}$\\
\hline
 14n20090 &  $(9_{47})^*(T)_c$, where $T=T_{1121}$\\

\end{tabular}\end{center}\end{table}
\end{example}

\section{The Jones polynomial of quasi-alternating links}
One of the most successful applications of quantum invariants of links  was obtained  by using  the Jones polynomial of alternating links to prove   long-lasting conjectures in classical knot theory \cite{kauffman1987state,Murasugi1987alternating,thistlethwaite1987spanning}. Indeed, the coefficients of the Jones polynomial of any  prime alternating link except $(2,n)$-torus links are alternating in sign. Moreover, the span of the polynomial is equal to the crossing number of the link and the polynomial has no gaps  \cite{thistlethwaite1987spanning}. In this section, we prove the following theorem:

\begin{theorem}\label{thm:V(t)}
Let $L'$ be a link obtained from a quasi-alternating link diagram $L$ by replacing  a quasi-alternating crossing $c$ by a reduced alternating  tangle $T$ of same type.
 If $V_L(t)$ has no gaps, then so is $V_{L'}(t)$.
\end{theorem}

We start by proving the  following proposition, where $[0]$ and $[\infty]$ denote the elementary trivial tangles in Fig~\ref{t1}.

\begin{proposition}\label{prop:1}
Let $c$ be a non-trivial crossing of an alternating tangle $T$ (not twisted tangle), and $T_0$ and $T_{\infty}$ be the tangles obtained from $T$ by smoothing crossing $c$ as shown in Fig.~\ref{1}(b) and \ref{1}(c), respectively. If $<T_0>=f_1(A)<[0]>+g_1(A)<[\infty]>$ and $<T_\infty>=f_2(A)<[0]>+g_2(A)<[\infty]>$, then the gap between $f_1(A)$ and $f_2(A)$, and  $g_1(A)$ and $g_2(A)$ is either two or six.
\end{proposition}

 \begin{figure}[!ht] {\centering
 \subfigure[{$[0]$}] {\includegraphics[scale=0.5]{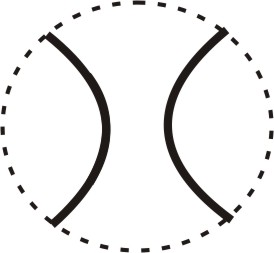}} \hspace{2cm}
\subfigure[{[$\infty$]}]{ \includegraphics[scale=0.5]{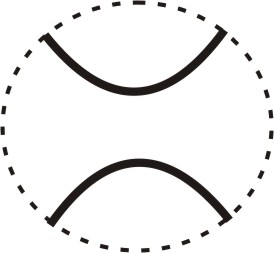}}\caption{Trivial tangles}\label{t1}}
\end{figure}

\begin{proof}
 First, notice that if $T$ is not reduced, then we have $<T>=(-A^3)^w<T_1>$ for some integer $w$ and a reduced alternating tangle $T_1$. Thus it would be enough to prove the result for reduced alternating tangles.

We will prove the result for positive tangles and similar  arguments can be used for negative tangles. It is clear that in  the case of a positive tangle, the unbounded region in the numerator closure  $N(T)$, is obtained by performing an $A$-smoothing at every crossing that is adjacent to  the  unbounded region.
 Since $T$ is not a twisted tangle, there exists a non-trivial crossing, say $c$, in $T$. Let $T_0$ and $T_\infty$ be the tangles obtained from $T$ by smoothing $c$ as illustrated in Fig.~\ref{1}(b) and \ref{1}(c), respectively. Then, both $T_0$ and $T_\infty$ are alternating and connected as $T$ is alternating and $c$ is non-trivial. Hence, links obtained by closing $T_0$ and $T_\infty$ are quasi-alternating, but the tangles need not be in their reduced forms. There are two cases to be considered.\\

\noindent{\it Case 1.} Suppose that $T_0$ and $T_\infty$ are reduced. Then the maximum and minimum degrees of $<N(T_0)>$ and $<N(T_\infty)>$ are given by their states with all $A$-smoothings and $A^{-1}$-smoothings, respectively. Let $S_0$ (respectively, $S_\infty$) and $S'_0$ (respectively, $S'_\infty$) be the states of $T_0$ (respectively, $T_\infty$) with all $A$-smoothings and $A^{-1}$-smoothings, respectively. As $T_0$ is obtained from $T$ by $A$-smoothing and  $T_\infty$ is  obtained by  $A^{-1}$-smoothing, then
\[|S_\infty|=|S_0|-1 \quad  \text{and} \quad |S'_0|=|S'_\infty|-1.\]
\begin{align*}\begin{split}  \text{Hence,~~}
\max \deg <N(T_0)>&=(n-1)+2|S_0|-2\\
&=(n-1)+2|S_\infty|-2 +2=\max \deg <N(T_\infty)>+2,
\end{split}\end{align*}
\begin{align*}\begin{split}
\text{and~~} \min \deg <N(T_0)>&=-(n-1)-2|S'_0|+2\\
&=-(n-1)-2|S'_\infty|+2 +2= \min \deg <N(T_\infty)>+2.
\end{split}\end{align*}
It is easy to observe that the contribution of states  $S_0$ and $S_\infty$ are in $f_1(A)<[0]>$ and $f_2(A)<[0]>$, while the contributions of  $S'_0$ and $S'_\infty$  are in $g_1(A)<[\infty]>$ and $g_2(A)<[\infty]>$, respectively. Therefore, the gap between $f_1(A)$ and $f_2(A)$, and between $g_1(A)$ and $g_2(A)$ is two.\\

\noindent{\it Case 2.} If $T_0$ or $T_\infty$ is reducible, then there exists two half twists in $T$ and crossing $c$ is one of these.
Let $T'$ be the tangle obtained from $T$ by contracting these two twists to a crossing $c'$, and let $T'_0$ and $T'_\infty$ be the tangles obtained from $T'$ by smoothing $c'$. We divide this case into the following subcases:\\

\noindent{\it Subcase 2a.} If $T$ has a horizontal  twist, then we have
\[<T_0>=-A^3<T'_0>  \text{and}  <T_\infty>=A<T'_0>+A^{-1}<T'_\infty>.\]
 If $c'$ is a nugatory crossing in $T'$, then either $T'_0$ or $T'_\infty$ is a split tangle. Since the determinant of a split link is zero and the determinant of a quasi-alternating link is $\geq1$, $T'_0$ should not be a split tangle.
Therefore by Remark~\ref{remk1}, we have
 \[<T_\infty>=(A+(-A^2-A^{-2})A^{-1})<T'_0>=-A^{-3}<T'_0>=A^{-6}<T_0>.\]
  Thus terms of $<T_0>$ are shifted by $A^{-6}$ in $<T_\infty>$.

If $c'$ is not a nugatory crossing, then it should be a quasi-alternating crossing in one of the closures of $T'$. Therefore, there is no cancellation between $A<T'_0>$ and $A^{-1}<T'_\infty>$. This implies that there is no cancellation between $-A^{-2}<T_0>$ and $A^{-1}<T'_\infty>$, and the terms of $<T_0>$ are shifted by $(-A^{-2})$ in $<T_\infty>$.

Hence, the result is true in this case.\\

\noindent{\it Subcase 2b.} If $T$ has a  vertical twist, then we have
\[<T_0>=A<T'_0>+A^{-1}<T'_\infty> \text{and} <T_\infty>=-A^{-3}<T'_\infty>.\]
 If $c'$ is a nugatory crossing in $T'$, then either $T'_0$ or $T'_\infty$ is a split tangle. Since the determinant of a split link is zero while the determinant of a  quasi-alternating link is $\geq 1$,
  $T'_\infty$ should not be a split tangle.
Therefore, by Remark~\ref{remk1}, we have
\[<T_0>=((-A^2-A^{-2})A+A^{-1})<T'_\infty>=-A^3<T'_\infty>=A^6<T_\infty>.\]
 Thus, the result is true in this case.

If  $c'$ is not a nugatory crossing, then one can use  similar arguments to those in  {\it subcase 2a}. This completes the proof of the proposition.
\end{proof}


\begin{lemma}\label{Lem:2}
Let $T$ be a reduced alternating positive tangle with
 $$<T>=Af(A)<[0]>+A^{-1}g(A)<[\infty]>.$$
  Then,  $f(A)$ and  $g(A)$ are alternating polynomials with same sign of coefficients, and they have  one nonzero term in  common. Moreover, the length of any gap in $f(A)$ and $g(A)$ is at most eight.
\end{lemma}
 \begin{proof}
 We prove this result by induction  on the number of crossings $n$ of the tangle $T$.
If  $n=1$, then   $f(A)=g(A)=1$ and the result is trivial. For $n=2$, $T$ is simply a twisted tangle. If $T$ is a vertically twisted tangle, then $f(A)=A$ and $g(A)=-A^{-3}+A$. Otherwise, $f(A)=-A^3+A^{-1}$ and $g(A)=A^{-1}$. In both  cases the  result is true.
Suppose that the result is true for all reduced alternating positive tangles with  less than $n$ crossings. Let $T$ be a reduced alternating positive tangle with $n$  crossings.
 We know that either $T$ is a twisted tangle or it has a non-trivial crossing. If $T$ is horizontally twisted, then we have $g(A)=A^{-n+1}$ and
\[ f(A)=A^{-n+1}-A^{-n+5}+A^{-n+9}\ldots+(-1)^{n-2}A^{3n-7}+(-1)^{n-1}A^{3n-3}.\]

If $T$ is vertically twisted, then we have $f(A)=A^{n-1}$ and
 \[g(A)=A^{n-1}-A^{n-5}\ldots+(-1)^{n-2}A^{-3n+7}+(-1)^{n-1}A^{-3n+3}.\]

If $T$ is not a twisted tangle, then it has a non-trivial crossing, say $c$, and
\[<T> =A<T_1>+A^{-1}<T_2>=Af(A)<[0]>+A^{-1}g(A)<[\infty]>,\]
 where $T_1$ and $T_2$ are the tangles obtained from $T$ by performing $A$-smoothing and  $A^{-1}$-smoothing, respectively at crossing $c$. For $i=1,2$, let $T'_i$ be the reduced diagram of $T_i$. For some integers $w_1$ and $w_2$, we have:
  \[<T_i>=(-A^3)^{w_i}<T'_i>=Af_i(A)<[0]>+A^{-1}g_i(A)<[\infty]>.\]
 Then  $f(A)= Af_1(A)+A^{-1}f_2(A)$ and $g(A)=Ag_1(A)+A^{-1}g_2(A)$.
Moreover, the result holds for $T_i$ as it is true for $T'_i$ by induction hypothesis. Hence, $f_i(A)$ and $g_i(A)$ are alternating, with same sign of coefficients, they have  gaps of length at most eight and have one nonzero term in common.

Since $f_i(A)$ and $g_i(A)$ are alternating and $c$ is a non-trivial crossing, it is easy to observe from the proof of Proposition~\ref{prop:1} that the coefficients of $Af_1(A)$ (respectively, $Ag_1(A)$) and $A^{-1}f_2(A)$ (respectively, $A^{-1}g_2(A)$) are of same signs.

By combining all these, we can see that $Af_1(A)+A^{-1}f_2(A)$ and $Ag_1(A)+A^{-1}g_2(A)$ are alternating with same sign of coefficients and have one nonzero term in common.

 At the end, there is no gap of length more than eight between $Af_1(A)$ (resp, $Ag_1(A)$) and $A^{-1}f_2(A)$ (resp, $A^{-1}g_2(A)$) as a result of Proposition~\ref{prop:1}.
 Hence, $f(A)$ and $g(A)$ satisfy  the  conditions stated in  Lemma~\ref{Lem:2}.
 \end{proof}

Now, we shall prove Theorem~\ref{thm:V(t)}.
Let $L$ be a quasi-alternating link diagram with no gap in its Jones polynomial and let $T$ be an extended tangle of a quasi-alternating crossing $c$ of $L$. We will assume that crossing $c$ is positive as the case of negative crossing can be proved similarly.
Then, we have
 \[<T>=Af(A)<[0]>+A^{-1}g(A)<[\infty]>,\]
where $f(A)$ and $g(A)$ are polynomials in $A$. Thus, one can write
\[<L'>=f(A)(A<L_0>)+g(A)(A^{-1}<L_\infty>).\]
Since $V_L(t)$ has no gap and $c$ is a quasi-alternating crossing in $L$, the gap between $A<L_0>$ and $A^{-1}<L_\infty>$ is at most 4 and both terms  are alternating with the same sign of coefficients.
By Lemma~\ref{Lem:2}, $f(A)$ and $g(A)$ are alternating with same sign of coefficients and have one nonzero term in common. Therefore, $f(A)(A<L_0>)$ and $g(A)(A^{-1}<L_\infty>)$ are alternating with same sign of coefficients and the gap between them  is no more than four.

The idea now is to show that  any gap of length  more than four in $f(A)(A<L_0>)$ can be filled by $g(A)(A^{-1}<L_\infty>)$ and vice-versa. Let us write
\[A<L_0>=\displaystyle \sum_{i=1}^{n_1} f_i(A) \quad \text{and} \quad A^{-1}<L_\infty>=\displaystyle \sum_{i=1}^{n_2} g_i(A),\]
where each of $f_i(A)$ and $g_i(A)$ are strictly alternating with
\[(\min \deg f_{i+1}(A)- \max \deg f_{i}(A))>4 \quad \text{and} \quad (\min \deg g_{i+1}(A)-\max \deg g_{i}(A))>4.\]
Then, for every $f_i(A)$ and $f_{i+1}(A)$, there is a $g_j(A)$ such that $f_i(A)+g_j(A)+f_{i+1}(A)$ is strictly alternating. Similarly, for every $g_i(A)$ and $g_{i+1}(A)$, there is a $f_j(A)$ such that $g_i(A)+f_j(A)+g_{i+1}(A)$ is strictly alternating.
We know that $f_i(A)$ and $g_i(A)$ are strictly alternating, and  $f(A)$ and $g(A)$ are alternating with gap of length at most eight. Therefore, it can be easily seen that $f(A)f_i(A)$ and $g(A)g_j(A)$ are strictly alternating with same sign of coefficients.

Any gap of length  more than four in $f(A)(A<L_0>)$, implies the existence of some $i$ such that the gap between $f(A)f_i(A)$ and $f(A)f_{i+1}(A)$ is more than four. We know that for any $f_i(A)$ and $f_{i+1}(A)$, there exists some $g_j(A)$ such that $f_i(A)+g_j(A)+f_{i+1}(A)$ is strictly alternating. Since $f(A)$ and $g(A)$ have  one nonzero common term and $f_i(A)+g_j(A)+f_{i+1}(A)$ has no gap of length more than four, $f(A)f_i(A)+g(A)g_j(A)+f(A)f_{i+1}(A)$ has no gap of length more than four and is strictly alternating.

Similarly, we can show that every gap between $g(A)g_i(A)$ and $g(A)g_{i+1}(A)$ is filled by some $f(A)f_j(A)$.
Hence, there is no gap in the bracket polynomial of $L'$  of length more than four.  Consequently, the  Jones polynomial of $L'$ has no gaps. This ends the proof of Theorem~\ref{thm:V(t)}.

\begin{corollary}Let $L$ be a quasi-alternating link with a quasi-alternating crossing $c$ and $T$ be an extended tangle of $c$. If length of any gap in $V_L(t)$ is at most $k$, then the Jones polynomial of the link that is obtained from $L$ by extending crossing $c$ to $T$ has no gap of length more than $k$.
\end{corollary}

\end{document}